\documentclass[reqno, english, 12pt]{amsart}
\pdfoutput = 1 
\usepackage[T1]{fontenc}

\usepackage[usenames, dvipsnames]{xcolor} 
\usepackage{amsfonts,amsmath,amssymb,amsthm,bbm,mathtools,comment}
\usepackage[shortlabels]{enumitem}
\usepackage{cancel}
\usepackage[hyphens]{url} 
\usepackage[linktoc = all, hidelinks, colorlinks, unicode=true, pagebackref=true]{hyperref} 
\usepackage[capitalise, compress, nameinlink, noabbrev]{cleveref} 
  \hypersetup{linkcolor={blue!70!black}, citecolor={green!70!black}, urlcolor={blue!70!black}}

\usepackage[mathscr]{euscript}
\usepackage{adjustbox,tikz,calc,graphics,babel,standalone}
\usetikzlibrary{decorations.pathmorphing,shapes,fit,positioning}
\usetikzlibrary{shapes.misc,calc,intersections,patterns,decorations.pathreplacing}
\usetikzlibrary{arrows,shapes,positioning}
\usetikzlibrary{decorations.markings,quotes,arrows.meta}
\usepackage[final]{microtype}
\usepackage[numbers]{natbib}
\usepackage{cmtiup}
\usepackage{graphicx}
\usepackage{caption}
\usepackage{subcaption}
\usepackage{verbatim}
\usepackage{array}
\usepackage[frame,cmtip,arrow,matrix,line,graph,curve]{xy}
\usepackage{graphpap,pstricks}
\usepackage{pifont}
\usepackage{todonotes}
\usetikzlibrary{topaths,calc} 
\tikzstyle{vertex} = [fill,shape=circle,node distance=80pt]
\tikzstyle{edge} = [opacity=0.4,fill opacity=0.0,line cap=round, line join=round, line width=40pt]
\tikzstyle{elabel} =  [fill,shape=circle,node distance=30pt]
\tikzstyle{circ} = [draw, fill=SpringGreen, circle, inner sep=0.12cm]
\tikzstyle{square} = [draw, fill=RedOrange, rectangle, inner sep=0.15cm]
\tikzstyle{triangle} = [draw, fill=Turquoise, regular polygon, regular polygon sides=3, inner sep=0.08cm]

\usepackage[margin = 2.3cm,foot=1cm]{geometry}

\allowdisplaybreaks

\theoremstyle{plain}
\newtheorem{theorem}{Theorem}[section]		
\newtheorem{lemma}[theorem]{Lemma}
\newtheorem{claim}[theorem]{Claim}
\newtheorem{proposition}[theorem]{Proposition}

\newtheorem{conjecture}[theorem]{Conjecture}

\newtheorem{definition}[theorem]{Definition}

\theoremstyle{remark}

\crefname{theorem}{theorem}{theorems}
\Crefname{theorem}{Theorem}{Theorems}

\crefname{lemma}{lemma}{lemmas}
\Crefname{lemma}{Lemma}{Lemmas}

\crefname{claim}{claim}{claims}
\Crefname{claim}{Claim}{Claims}

\crefname{proposition}{proposition}{propositions}
\Crefname{proposition}{Proposition}{Propositions}

\crefname{corollary}{corollary}{corollaries}
\Crefname{corollary}{Corollary}{Corollaries}

\crefname{conjecture}{conjecture}{conjectures}
\Crefname{conjecture}{Conjecture}{Conjectures}

\crefname{problem}{problem}{problems}
\Crefname{problem}{Problem}{Problems}

\crefname{observation}{observation}{observations}
\Crefname{observation}{Observation}{Observations}

\crefname{example}{example}{examples}
\Crefname{example}{Example}{Examples}

\crefname{question}{question}{questions}
\Crefname{question}{Question}{Questions}

\crefname{definition}{definition}{definitions}
\Crefname{definition}{Definition}{Definitions}

\crefname{construction}{construction}{constructions}
\Crefname{construction}{Construction}{Constructions}

\renewcommand{\emptyset}{\varnothing}
\renewcommand{\le}{\leqslant}
\renewcommand{\leq}{\leqslant}
\renewcommand{\ge}{\geqslant}
\renewcommand{\geq}{\geqslant}

\let\originalleft\left
\let\originalright\right
\renewcommand{\left}{\mathopen{}\mathclose\bgroup\originalleft}
\renewcommand{\right}{\aftergroup\egroup\originalright}

\makeatletter
\def\imod#1{\allowbreak\mkern10mu({\operator@font mod}\,\,#1)}
\makeatother

\author{
 Nemanja Dragani\'c \and Ant\'onio Gir\~ao }
\thanks{
AG: Department of Mathematics, University College London, London, WC1E 6BT, UK.
E-mail: \texttt{a.girao@ucl.ac.uk}\\
ND: Mathematical Institute, University of Oxford, Oxford OX2 6GG, UK. 
E-mail: \texttt{nemanja.draganic\allowbreak@maths.ox.ac.uk}
}

\begin{document}

\title{Cycles with almost linearly many chords}

\begin{abstract}
We prove that constant minimum degree already forces cycles with almost linearly many chords. Specifically, every graph $G$ with $\delta(G)\ge C$ contains a cycle of length $\ell\ge 4$ with $\Omega(\ell/\log^{C}\ell)$ chords for some absolute constant $C>0$. This is the first result showing that a constant-degree condition yields an unbounded---indeed nearly linear---number of chords, placing our bound within a polylogarithmic factor of the Chen--Erd\H{o}s--Staton conjecture. It also gives a strong affirmative conclusion in the direction of a recent question of Dvo\v{r}\'ak, Martins, Thomass\'e, and Trotignon asking whether constant-degree graphs must contain cycles whose chord counts grow with their length. 
\end{abstract}

\maketitle

\section{Introduction}
A central theme in extremal graph theory is understanding how many edges a graph on $n$ vertices must have in order to force the appearance of particular substructures. In the sparse regime, especially when the average degree is constant or quasi-constant, remarkable progress has been achieved over the past decade. A unifying principle behind many of these developments is that of \emph{robust sublinear expansion}: one typically extracts from the original graph a mildly expanding subgraph (neighborhoods of vertex sets grow by a sublinear factor) and then exploits this expansion to find structure.

 Liu and Montgomery~\cite{LiuMontgomery2023,liu2023solution} proved that sufficiently large constant average degree forces a cycle whose length is a power of two, resolving a 1984 conjecture of Erd\H{o}s~\cite{eirdos1984some}; in the same work they also settled the Odd-cycle problem of Erd\H{o}s and Hajnal~\cite{erdos1981problems} and established the existence of large clique subdivisions where each edge is subdivided the same number of times, answering a question of Thomassen~\cite{thomassen1984subdivisions}. Further results of Fern\'andez, Kim, Kim, and Liu~\cite{fernandez2022nested} showed that every constant-average-degree graph contains two nested, edge-disjoint cycles preserving cyclic order, resolving another problem of Erd\H{o}s~\cite{erdHos1975problems}. Subsequent work demonstrated the existence of \emph{pillars}—two vertex-disjoint cycles of equal length joined by vertex-disjoint paths of the same length—again under the same degree assumptions~\cite{fernandez2023build}. 

Other classical problems require more than linearly many edges. A question of Erd\H{o}s~\cite{erdHos1975problems} asks for the average degree needed to force two edge-disjoint cycles on the same vertex set. A result of Pyber, Rödl, and Szemerédi~\cite{pyber1995dense} implies that guaranteeing any $4$-regular subgraph already requires $\Omega(n\log\log n)$ edges, while Chakraborti, Janzer, Methuku, and Montgomery~\cite{chakraborti2025edge} recently showed that for the Erd\H{o}s question, an average degree of $(\log n)^C$ suffices, for some large constant $C>0$.

A 1996 result of Chen, Erd\H{o}s, and Staton~\cite{chen1996proof}, resolving a problem of Bollob\'as, states that for every $k\in\mathbb{N}$ there exists a constant $c_k$ such that any graph with average degree at least $c_k$ contains cycles $C_1,\ldots,C_k$ where each $C_{i+1}$ is formed entirely by chords of $C_i$. In a related direction, Thomassen proved that for every $k$ there exists $g_k$ such that any graph with minimum degree~$3$ and girth at least $g_k$ contains a cycle with at least $k$ chords. Motivated by these results, Chen, Erd\H{o}s, and Staton asked in 1996 how many edges are required in an $n$-vertex graph to guarantee a cycle with as many chords as vertices. The best current bound, due to the first author together with Methuku, Munh\'a Correia, and Sudakov~\cite{draganic2024cycles}, shows that average degree at most $(\log n)^{8}$ already suffices; intriguingly, even the relaxed version of seeking a cycle of length $\ell$ with a linear number of chords, say $\varepsilon \ell$, is still wide open.

In this direction, Dvo\v{r}ák, Martins, Thomass\'e, and Trotignon~\cite{dvovrak2025lollipops} further relaxed the condition on the number of edges, asking whether there exists a function $f\colon \mathbb{N}\to\mathbb{N}$ with $f(\ell)\to\infty$ as $\ell\to\infty$ such that every graph with minimum degree~$3$ contains a cycle on $\ell$ vertices with $f(\ell)$ chords.
Assuming only a constant average degree, we answer this question in a strong form, while simultaneously coming close to resolving the Chen–Erd\H{o}s–Staton conjecture, providing a lower bound on the number of chords which is optimal up to a logarithmic factor.

\begin{theorem}\label{thm:main}
There exists constants $c,C>0$ such  that 
every graph $G$ with $\delta(G)\ge C$ contains a cycle of length $\ell$ with at least $\Omega\left(\frac{\ell}{\log^{c}\ell}\right)$ chords, for some positive integer $\ell$.
\end{theorem}
 
\section{Preliminaries}
\paragraph{\textbf{Notation.}} 
We follow standard graph-theoretic conventions. For a graph $G$, we write $V(G)$ and $E(G)$
for its vertex and edge sets, $d(G)$ for its average degree, and $\delta(G)$ and $\Delta(G)$
for its minimum and maximum degrees. For a vertex set $X \subseteq V(G)$, we denote by
$N_G(X)$ the set of all vertices of $G$ outside of $X$ adjacent to at least one vertex of $X$; when the
underlying graph is clear, we simply write $N(X)$. A \emph{spider} is a tree formed by a
collection of internally vertex-disjoint paths that all share a common endpoint $v$, called
the \emph{center}, and are otherwise disjoint. All logarithms are base~2 unless otherwise specified.

We will need the well known decomposition of connected graphs. 

\begin{proposition}[Block--cut structure]
Every connected graph \(G\) admits a unique decomposition into \emph{blocks}, that is, maximal \(2\)-connected subgraphs (and bridges). 
Distinct blocks intersect in at most one vertex, and the incidence structure between blocks and cut-vertices forms a tree, called the \emph{block--cut tree} of \(G\).
\end{proposition}

We will also use the following result by Kuhn and Osthus~\cite{kuhn2004every}, which guarantees a $C_4$-free subgraph of large average degree in graphs with constant average degree. The bound on the required initial degree was later refined by Montgomery, Pokrovskiy, and Sudakov\cite{montgomery2021c}.
\begin{theorem}\label{thm:c4free}
    For every $k>0$ there is a $d>0$, such that every graph of average degree at least $d$ contains a subgraph of average degree $k$ which is $C_4$-free.
\end{theorem}

\begin{definition}
    A graph $G$ is an $\alpha$-expander if every subset $S$ of vertices of size at most $|G|/2$ has $|N(S)|\geq\alpha|S|$.
\end{definition}
We will also need the following result of Friedman and Krivelevich~\cite{friedman2021cycle}.
\begin{theorem}\label{thm: long cycle}
    Every $n$-vertex $\alpha$-expander contains a cycle of length at least $\Omega\left(\frac{\alpha^3}{\log(1/\alpha)}\right) n$.
\end{theorem}
\subsection{Sublinear expanders}

We begin with the standard definition of sublinear expanders.
\begin{definition}[Sublinear expander]\label{def:sublinear-expander}
Let $\varepsilon_1 > 0$ and $k \in \mathbb{N}$. A graph $G$ is an $(\varepsilon_1, k)$\emph{-expander} 
if for all $X \subset V(G)$ with $k/2 \le |X| \le |G|/2$, and any subgraph $F \subseteq G$ 
with $e(F) \le d(G)\,\varepsilon(|X|)\,|X|$, we have
\[
|N_{G \setminus F}(X)| \ge \varepsilon(|X|)\,|X|,
\]
where
\[
\varepsilon(x) = \varepsilon(x, \varepsilon_1, k) =
\begin{cases}
0, & \text{if } x < k/5,\\[4pt]
\displaystyle \frac{\varepsilon_1}{\log^2 (15x/k)}, & \text{if } x \ge k/5.
\end{cases}
\]
\end{definition}
For our purposes, the subgraph $F$ from the definition will always be chosen to be empty, as our proof does not require robustness of expansion.
We now state the classical result of Komlós and Szemerédi~\cite{komlos1996topological}, which ensures the existence of a robust sublinear expander as a subgraph, with only a small loss in average degree, and with a bound on the minimum degree.

\begin{theorem}\label{thm:expander-subgraph}
There exists some $\varepsilon_1 > 0$ such that the following holds for every $k > 0$.
Every graph $G$ has an $(\varepsilon_1, k)$-expander subgraph $H$ with
\[
d(H) \ge \frac{d(G)}{2}
\quad \text{and} \quad
\delta(H) \ge \frac{d(H)}{2}.
\]
\end{theorem}

The following result shows that there is a short path between two sets that avoids another small set; the proof follows from a simple greedy exploration of the graph, so we omit it. 
\begin{lemma}\label{lem:connecting-path}
    Let $G$ be a $\alpha$-expander for some $\alpha>0$. Let $X,Y,B$ be disjoint sets. If $|X|,|Y|>2|B|/\alpha$, then there is a path between $X$ and $Y$ which avoids $B$ and has length at most $2\log_{1+\alpha/2}n$.
\end{lemma}

The next result shows that if a small subset of vertices of an expander is removed, we can remove a few more vertices to obtain a graph with similar expansion properties.

\begin{lemma}\label{lem:clean for expansion}
    Let $G$ be a $n$-vertex $\alpha$-expander for $0<\alpha<1/100$, and let $U$ be a subset of vertices of size $|U|\leq \alpha^2 n/100$. Then there is a subset of vertices $B$ of size at most $2|U|/\alpha$ with $|N_{G\setminus U}(B)|\leq |B|$, such that $G\setminus(U\cup B)$ is an $\alpha/2$-expander.
\end{lemma}
\begin{proof}
    Let $B$ the largest subset of $V(G)\setminus U$ of size at most $n/2$, for which $|N_{G\setminus U}(B)|< \alpha |B|/2$. We claim that this set satisfies the claim.
    First, $B$ is relatively small; indeed, if $|B|\geq 2|U|/\alpha$, then in $G$ we would have $|N_G(B)|\geq \alpha|B|\geq  2|U|$. Thus $|N_{G\setminus U}(B)|\geq \alpha|B|-|U|\geq \alpha|B|/2$, a contradiction with the definition of $B$. Hence $|B|\leq 2|U|/\alpha\leq \alpha n/50.$
    
    Now we show that $G':=G\setminus(U\cup B)$ is an $\alpha/2$-expander. Indeed, otherwise there is a set of size $|X|\leq |G'|/2$ such that $|N_{G'}(X)|\leq \alpha|X|/2$.
    We have two cases: if $|X\cup B|<|G'|/2$ then consider $Y:=X\cup B$ to get a contradiction with the maximality of $B$, as $|N_{G\setminus U}(Y)|\leq \alpha|X\cup B|/2$. Otherwise consider a subset $Y\subseteq X\cup B$ of size $|G'|/2$.
    Then we have $$|N_G(Y)|\leq |N_{G\setminus U}(X)|+|N_{G\setminus U}( B)|+ |U|+|(X\cup B)\setminus Y|\leq \alpha(|X|+|B|)/2+|U|+|B|\leq \alpha n/3,$$ whereas by initial expansion of $G$ we would need to have $|N_G(Y)|\geq \alpha|G'|/2$, a contradiction.
\end{proof}

\section{Cycle extenders}
The goal of this section is to prove \Cref{lem:cycle extender} below---this is a result that shows the existence of an appropriate gadget given that the host graph is mildly expanding, and is $2$-connected.
Before that, we prove several supporting results.

\begin{lemma}\label{lem:interlacing-chords}
    If $G$ has minimum degree $10$ then it contains a cycle with two interlacing chords. 
\end{lemma}

\begin{proof}
    Take a longest path \( xPy \) in the graph. Perform Pósa rotations starting from the endpoint \( x \), and let \( A \) be the set of all possible endpoints obtained through these rotations. Let \( w \in A \) be the vertex that lies closest to \( y \) along the original path \( xPy \).
Then there exists a cycle \( C \) containing all vertices in the segment of \( xPy \) between \( x \) and \( w \). Moreover, all vertices of \( A \) lie within this segment. By Pósa’s lemma, every neighbour of a vertex in \( A \) is contained either in \( A \) itself or in the neighbourhood \( N(A) \) of \( A \) along the path \( xPy \). Consequently, the induced subgraph \( G[N(A) \cup A] \) has at most \( 3|A| \) vertices and minimum degree at least \( 10 \).
Since such a graph is not planar, if we embed the vertices of \( C \) on a circle in the plane and draw the edges as straight-line chords, there must exist at least one pair of crossing edges—corresponding to interlacing chords of \( C \).

\end{proof}

\begin{proposition}\label{prop:cycle extend}
    Let $G$ be a $2$-connected graph with a cycle $C$ containing two interlacing chords.
    Let $C'$ be another vertex disjoint cycle. Then $G$ contains a chorded cycle of size at least $|C'|/2$. 
\end{proposition}

\begin{proof}
     By Menger's theorem, there are two vertex disjoint paths from $V(C')$ to $V(C)$. In each case depending on where the endpoints of the paths are in $C$, as shown in \Cref{fig:long chorded cycle}, we get a chorded cycle of length at least $|C'|/2$.
\end{proof}

\begin{figure}[h]
\centering
\begin{tabular}{c@{\hspace{1.5cm}}c}

\begin{tikzpicture}[scale=1]
  \coordinate (C1) at (0,0);
  \coordinate (C2) at (4,0);
  \def\rone{1.4}
  \def\rtwo{1.4}

  \draw[line width=0.8pt] (C1) circle (\rone);
  \draw[line width=0.8pt] (C2) circle (\rtwo);

  \coordinate (A) at ($(C1)+(240:\rone)$);
  \coordinate (B) at ($(C1)+(60:\rone)$);
  \coordinate (C) at ($(C1)+(300:\rone)$);
  \coordinate (D) at ($(C1)+(120:\rone)$);

  \coordinate (R) at ($(C1)+(90:\rone)$);
  \coordinate (Q) at ($(C1)+(270:\rone)$);

  \coordinate (E) at ($(C2)+(150:\rtwo)$);
  \coordinate (F) at ($(C2)+(270:\rtwo)$);

  \draw[line width=0.9pt] (A) -- (B);
  \draw[line width=0.9pt] (C) -- (D);

  \draw[line width=3.6pt, line cap=round, blue, bend right=30] (E) to (R);
  \draw[line width=3.6pt, line cap=round, blue, bend left=30] (F) to (Q);
  \draw[line width=3.6pt, line cap=round, blue] (F) arc (-90:150:\rtwo);
  \draw[line width=3.6pt, line cap=round, blue] ($(C1)+(90:\rone)$) arc (90:240:\rone);
  \draw[line width=3.6pt, line cap=round, blue] (A) -- (B);
  \draw[line width=3.6pt, line cap=round, blue] ($(C1)+(270:\rone)$) arc (-90:60:\rone);

  \draw[line width=0.8pt] (C1) circle (\rone);
  \draw[line width=0.8pt] (C2) circle (\rtwo);

  \foreach \p in {A,B,C,D,E,F,R,Q}
    \fill (\p) circle (3.2pt);
\end{tikzpicture}
&
\begin{tikzpicture}[scale=1]
  \coordinate (C1) at (0,0);
  \coordinate (C2) at (4,0);
  \def\rone{1.4}
  \def\rtwo{1.4}

  \draw[line width=0.8pt] (C1) circle (\rone);
  \draw[line width=0.8pt] (C2) circle (\rtwo);

  \coordinate (A) at ($(C1)+(240:\rone)$);
  \coordinate (B) at ($(C1)+(60:\rone)$);
  \coordinate (C) at ($(C1)+(300:\rone)$);
  \coordinate (D) at ($(C1)+(120:\rone)$);

  \coordinate (R) at ($(C1)+(0:\rone)$);      
  \coordinate (Q) at ($(C1)+(270:\rone)$);

  \coordinate (E) at ($(C2)+(150:\rtwo)$);
  \coordinate (F) at ($(C2)+(270:\rtwo)$);

  \draw[line width=0.9pt] (A) -- (B);
  \draw[line width=0.9pt] (C) -- (D);

  \draw[line width=3.6pt, line cap=round, blue, bend right=30] (E) to (R);
  \draw[line width=3.6pt, line cap=round, blue, bend left=30] (F) to (Q);
  \draw[line width=3.6pt, line cap=round, blue] (F) arc (-90:150:\rtwo);

  \draw[line width=3.6pt, line cap=round, blue]
    ($(C1)+(0:\rone)$) arc (0:270:\rone);

  \draw[line width=0.8pt] (C1) circle (\rone);
  \draw[line width=0.8pt] (C2) circle (\rtwo);

  \foreach \p in {A,B,C,D,E,F,R,Q}
    \fill (\p) circle (3.2pt);
\end{tikzpicture}

\end{tabular}
\caption{Depending on where the vertex-disjoint paths meet the first cycle (either on opposite sides of each chord or not), we obtain the two configurations shown. In both cases, the blue cycle contains a chord and includes at least half of the vertices of the second cycle.}
\label{fig:long chorded cycle}
\end{figure}
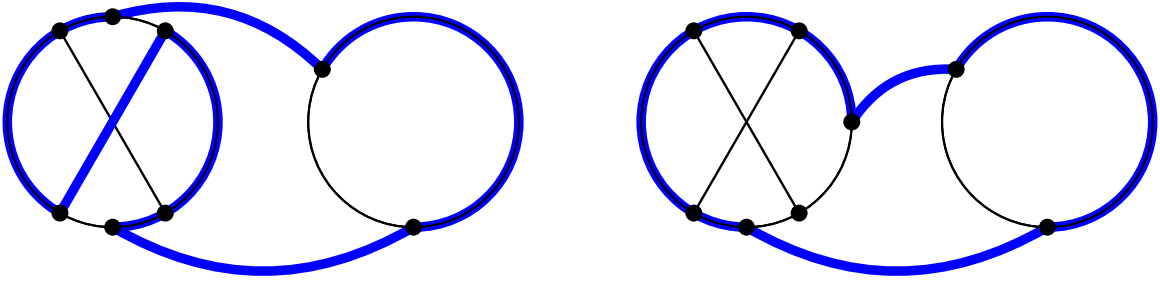

\begin{proposition}\label{prop:small diameter sets}
    Let $G$ be a graph on $n$ vertices which is an $\frac{1}{\log^5(n)}$-expander for large enough $n$. Then, for all $n\geq m\geq \log^{8}(n)$, there is $S\subset V(G)$,  of size $m$ and diameter at most $\log^7(n)$.
\end{proposition}
\begin{proof}
   Clearly, $G$ has diameter at most $\log^{7}(n)$, so it contains a tree of this depth. Remove leaves until the tree reaches the required size.
\end{proof}

\begin{lemma}\label{lem:shorten cycle}
    Let $G$ be a $\frac{1}{\log^5(n)}$-expander on $n$ vertices for large enough $n$. 
Let $C\subset V(G)$ be a cycle of length $m\geq \log^{30}(n)$ with a chord. 
Then, we can find a cycle $C'$ of length $\log^{8}(n)\leq \ell \leq \log^{30}(n)$ with a chord. 
\end{lemma}
\begin{proof}
Let $C$ be a shortest cycle of length at least $\log^{8}(n)$ which contains a chord, and suppose it is of length at least $\log^{30}n$.
First note that the chord splits the cycle into two paths of length at least $\log^{12}{n}$, as otherwise if one of the paths, call it $M$, is shorter, we can shorten the other by using \Cref{lem:connecting-path}.
Indeed, if the longer path $L$ consists of consecutive segments $L_1,L_2,L_3$, where $|L_1|=|L_2|=\frac{|L|-\log^{10}n}{2}\geq \log^{29}n$ and $|L_2|=\log^{10}n$, then there is a path of length at most $\log^8 n$ between $L_1$ and $L_3$ which avoids $L_2\cup M$, thus giving a shorter chorded cycle.

Consider two arbitrary vertices $x,y$ distinct from the endpoints of the chord that are at the largest distance in $C$, and let $P_1$ and $P_2$ be the paths with endpoints $x,y$ in $C$. Let $B$ be the set of vertices at distance at most $\log^{15} n$ from $x,$ or $y$. By \Cref{lem:connecting-path} there is a path of length at most $\log^7n$ between $P_1$ and $P_2$ which contains no vertices in $B$; let $w,z$ be the endpoints of a shortest such path $Q_1$, and $B'$ the vertices at distance at most $\log^9 n$ from $w,z$. Consider the paths $P_3,P_4$ in $C$ with endpoints $w,z$, and note that $|P_3|,|P_4|\geq |B|/2\geq \log^{15 }n$. Now consider the shortest path $Q_2$ between $P_3$ and $P_4$ which avoids $B'$, and note that $|Q_2|\leq \log^7 n$ as well.

If either of $Q_1$ or $Q_2$ is on the same side of the chord, we get a contradiction by getting a shorter cycle with a chord; indeed we can replace the interval between the endpoints of $Q_i$ in $C$ by the path $Q_i$ --- the interval contains either half of $B$ or half of $B'$, which are of size at least $\log^8n$, while $Q_i$ is of length at most $\log^7 n$.
On the other hand, if both of the paths cross the chord, we can use both $Q_1$ and $Q_2$ instead of the two intervals in $C$ whose endpoints are among $x,y,w,z$ which do not contain any endpoints of the chord --- these again contain half of $B'$, so we are done.

    
\end{proof}

\begin{lemma}\label{lem:routing to big sets}
    Let $G$ be a $n$-vertex $1/\log^{5}(n)$-expander for large $n$. 
    Let $C$ be a chorded cycle of length between $\log^{30}(n)$ and $\log^{50}(n)$. Disjoint from it, let $A_1,A_2,A_3$ be three connected, vertex disjoint sets of size at least $s\geq \log^{50}(n)$ each of which with diameter at most $\log^{8}(n)$. Then for two of those sets there exists connected subsets $A_i'\subseteq A_i$ and $A_j'\subseteq A_j$ of sizes at least $s/2$ with two vertex disjoint paths of length at most $\log^{7}(n)$ from $C$ to $A_i',A_j'$ whose initial vertices on the same side of the chord on $C$. 
\end{lemma}
\begin{proof}
   By shrinking we may assume $|A_1|=|A_2|=|A_3|=s$. Let $x_1\in A_1$, $x\in A_2$ and $x_3\in A_3$ be three arbitrary vertices. Let $T_1,T_2, T_3$ be three spanning trees $T_i\in G[A_i]$ and $T_i$ is rooted at $x_i$ and is of diameter $\log^8(n)$ 
   For each $T_i$, we define $B_i\subset V(T_i)$ a set  of \textit{dangerous}
vertices --- $y\in T_i$ is dangerous if it is not a leaf, and by deleting it the component not containing $x_i$ has size at least $s/\log^{10}(n)$. 

First we show that $B_i$ is of size at most $\log^{18}(n)$. Note that every set $D\subset V(T_i)$ of dangerous vertices in which no vertex is a ancestor of another is of size at most $\log^{10}(n)$. Indeed, for $u,v\in D$ by assumption the component in $T_i-v$ which does not contain $x_i$ is disjoint from the component in $T_i-u$ which does not contain $x_i$; hence for the total size of those components to be less than $n$, we have $|D|\leq \log^{10}(n)$.
By assumption there are at most $\log^{8}(n)$ ancestors of a given vertex, as this is a bound on the depth of the tree.
Hence in total there are at most $
\log^{18}(n)$ dangerous vertices in $T_i$. 

We now find a path $P_1$ from $A_1$ to $C$ of size at most $\log^{7}(n)$ avoiding $B_2\cup B_3$, by Lemma~\ref{lem:connecting-path}. We may assume $P_1$ has exactly one vertex $y_1$ in $A_1$. Let $Q_1$ be the path in $T_1$ from $y_1$ to $x_1$. Similarly, we find a path $P_2$ from $A_2$ to $C$ avoiding $V(P_1)\cup V(Q_1)\cup B_1\cup B_2$ of length at most $\log^{7}(n)$. As before, let $Q_2$ be the path in $T_2$ from the first vertex of $P_2$ to $x_2$.  Finally, we find a path $P_3$ from $A_3$ to $C$ avoiding $V(P_1)\cup V(P_2) \cup V(Q_1)\cup V(Q_2)\cup B_1\cup B_2$. 
By construction, we have three pairwise vertex disjoint paths $P_i$ from $A_i$ to $C_i$. By pigeonhole, we may assume two of them say $P_1,P_2$ end on the same side of the chord in $C$. 
Finally, note that by deleting $V(P_2)\cap A_1$, the component of $x_1$ has size at least $s- |P_1|s/\log^{10}(n)\geq s/2$. The same holds for $A_2$, as we wanted to show.

\end{proof}

Finally, we need the following definition to state the main result of this section.

\begin{definition}\label{def:cycle extender}
    A subgraph $F$ of an $n$ vertex graph is a \emph{cycle extender} if $F$ is the union of the following graphs (see \Cref{fig:gadgets}):
         \begin{itemize}
             \item A cycle of length at most $\log^{30} n$.
             \item Two disjoint paths, $P_1$ and $P_2$ of length at most $2\log^{30}n$, such that their endpoints are consecutive vertices in $C$, but they are otherwise disjoint from $C$.
             \item Two disjoint sets $A_1,A_2$ of diameter at most $\log^{7}(n)$ and size $n^{1/4}$, where each $A_i$ contains the other endpoint of $P_i$, but is otherwise disjoint from $C\cup P_1 \cup P_2$.
         \end{itemize}
\end{definition}

We can now state the main result of this section. 
\begin{lemma}\label{lem:cycle extender}
    Let $G$ be a $2$-connected $n$-vertex graph that is an $1/\log^5{n}$-expander and has average degree at least $20$, and $n$ is large enough. 
    Then $G$ contains a cycle extender.
\end{lemma}

\begin{proof}
    Let $G'$ be a subgraph with minimum degree $10$. By \Cref{lem:interlacing-chords} there is a cycle $C$ with interlacing chords in $G'$ and thus in $G$ as well. If $C$ is not already of length say $\sqrt n$, by \Cref{lem:clean for expansion} we can remove $V(C)$ and a set $B$ of size $|B|\leq n^{4/5}$ from $G$, to obtain a $1/2\log^{5} n$-expander $G''$. By \Cref{thm: long cycle} this graph contains a cycle $C'$ of length at least $n/\log^{16} n$.

    By \Cref{prop:cycle extend}, we thus get a chorded cycle in $G$ of length at least $n/(2\log^{16} n)$. We now apply \Cref{lem:shorten cycle} to get a chorded cycle $Q$ of length between $\log^8 n$ and $\log^{30} n$. Now, using the expansion property, \Cref{lem:clean for expansion} and \Cref{prop:small diameter sets} we can get three disjoint sets $A_1,A_2,A_3$ of small diameter (at most $\log^8 n$) and size $n^{2/3}$ disjoint from $Q$. Thus we can use \Cref{lem:routing to big sets} to connect $Q$ via two paths $P_1,P_2$ of length at most $\log^7 n$ to large connected subsets, say $A_1' \subseteq A_1$ and $A_2' \subseteq A_2$ of sizes $\sqrt n$, such that the endpoints of the paths in $Q$ are on the same side of the chord. Denote by $P$ the path in $Q$ between those two endpoints, and which is on the same side of the chord. Now, $Q\cup P_1\cup P_2\cup A_1'\cup A_2'$ without the internal vertices of $P$ is the required cycle extender.
\end{proof}

\section{Cycles with many chords}
    We are ready to prove \Cref{thm:main}.
\begin{proof}[Proof of Theorem $1.1$]

Let $\varepsilon_1>0$ be given by \Cref{thm:expander-subgraph}, and chose $k=1$. Let $C_0(\varepsilon_1)$ be large enough, and assume $G$ has average degree at least $C\geq C_0$. We may assume that $G$ is $C_4$-free by \Cref{thm:c4free}.
Pass to a $(\varepsilon_1, k)$-robust-expander subgraph $H\subset G$ with $\delta(H)\geq d$ where $d$ is still large enough compared to $\varepsilon_1$. Suppose $H$ has $n$ vertices. 
Assume the contrary, that there does not exists a cycle with many chords.


\subsection{Gadgets and how we use them}
Fix \( m = 2^{\log^{1/4}(n)} \), and let \( L \) be the set of vertices with degree at least \( m \). 
There are two basic kinds of structures we hope to find, and depending on the nature of \( H \), we will argue that many such structures appear.
\begin{figure}[h]
\begin{center}
\begin{tikzpicture}[scale=1, every node/.style={font=\small},
    cluster/.style={draw, rounded corners=8pt, dashed, inner sep=15pt},
    vert/.style={circle,draw,inner sep=2pt,minimum size=7pt,fill=white},
    smallp/.style={line width=1pt},
    pathstyle/.style={decorate, decoration={snake, amplitude=0.6mm, segment length=3mm}, line width=1pt}
    ]
\usetikzlibrary{decorations.pathmorphing}

\begin{scope}[xshift=-5cm]
  \node[vert] (x) at (-1,0) {$x$};
  \node[vert] (z1) at (1.6, 0.9) {$z_1$};
  \node[vert] (z2) at (1.3, 0)   {$z_2$};
  \node[vert] (z3) at (1.6,-0.9) {$z_3$};

  \draw[pathstyle] (x) .. controls (-0.5,0.9) and (0.5,1.1) .. (z1); 
  \draw[smallp] (x) -- (z2);                                        
  \draw[pathstyle] (x) .. controls (-0.5,-0.9) and (0.5,-1.1) .. (z3); 

  \node[cluster, fit=(z1) (z2) (z3)] (L) {};
  \node[font=\small] at (0.8,1.5) {$L$};
\end{scope}

\begin{scope}[xshift=4.5cm]
  \def\R{1.6}
  \foreach \i in {1,...,8} {
    \pgfmathsetmacro{\ang}{90 - (\i-1)*360/8}
    \node[vert] (c\i) at ({\R*cos(\ang)},{\R*sin(\ang)}) {};
  }
  \foreach \i/\j in {1/2,2/3,3/4,4/5,5/6,6/7,7/8,8/1} {
    \draw[smallp] (c\i) -- (c\j);
  }
  \node[font=\small] at (0,-2.1) {$C$};

  \coordinate (a1center) at ($(c2)+(2.6,2.0)$);
  \coordinate (a2center) at ($(c3)+(2.6,-2.0)$);

  \coordinate (p1end) at ($(a1center)+(-0.5,-0.3)$);
  \coordinate (p2end) at ($(a2center)+(-0.5,0.3)$);

  \draw[pathstyle] (c2) .. controls ($(c2)+(0.8,1.4)$) and ($(p1end)+(-0.6,0.6)$) .. (p1end);
  \draw[pathstyle] (c3) .. controls ($(c3)+(0.8,-1.4)$) and ($(p2end)+(-0.6,-0.6)$) .. (p2end);

  \node[font=\small] at ($(c2)!0.65!(p1end)+(0,0.25)$) {$P_1$};
  \node[font=\small] at ($(c3)!0.65!(p2end)+(0,-0.25)$) {$P_2$};

  \node[cluster, fit=(a1center), label=above:{$A_1$}] {};
  \node[cluster, fit=(a2center), label=below:{$A_2$}] {};
\end{scope}

\end{tikzpicture}
\caption{The two types of useful structures: a nice spider on the left, and a cycle extender on the right}\label{fig:gadgets}
\end{center}
\end{figure}
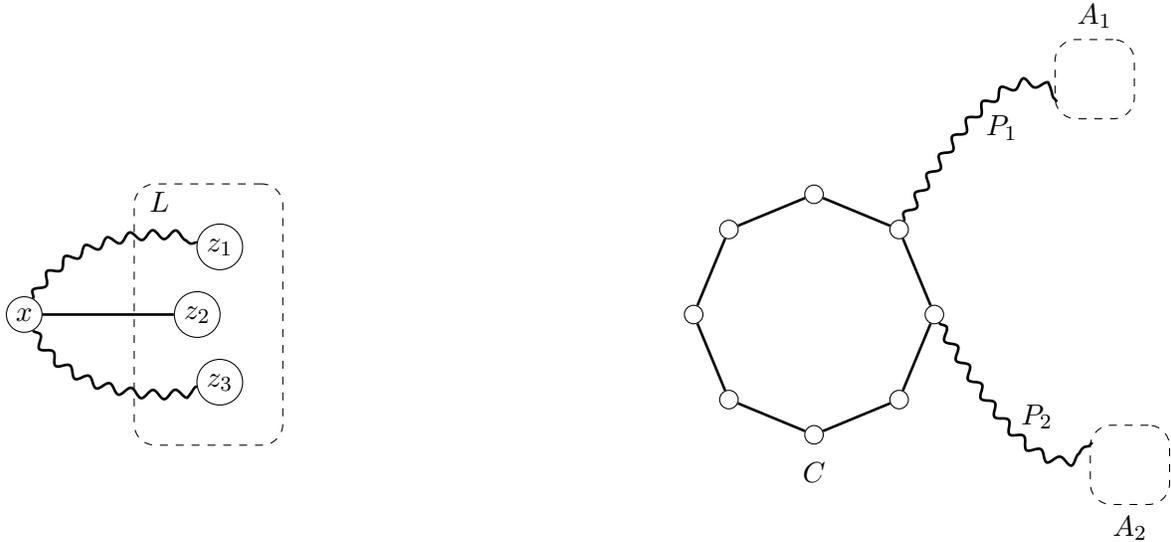

     \begin{enumerate}[label=\textbf{Type~\arabic*:}]
         \item A spider graph $S$ with center $x$ and three leaves $Z=\{z_1,z_2,z_3\}$ is \emph{a nice spider} if $Z\subset L$, the path between $x$ and $z_2$ is of length one, and the other two paths are of length at most $\log^8 n$.
         \item A cycle extender (see \Cref{def:cycle extender}).
     \end{enumerate}

     \begin{lemma}\label{lem:connecting gadgets}
    If $H$ contains $2^{\log^{1/100}(n)}$ vertex disjoint copies of graphs that are either a nice spider or a cycle extender, then for some $\ell>0$ it contains a cycle of lenth at least $\ell$, with at least $\ell/\log^{1000}\ell$ many chords. 
\end{lemma}
\begin{proof}
Denote $s:=2^{\log^{1/100}n}$.
    Let $\{F_i\}_{i\in[s]} $ be the collection of gadgets we have at our disposal. If $F_i$ is a nice spider with leaves $\{z_1,z_2,z_3\}$, denote by $N_i^a$ the neighbourhood of $z_a$ for $a\in[3]$.
    Furthermore, split $N_i^2$ into two equal parts $L_i$ and $R_i$.
    For cycle extenders $F_i$ denote their sets of size $n^{1/4}$ with $A^1_i$ and $A_i^2$.

    If we can find vertex disjoint paths of length at most $\log^7n$ as follows, it is easy to see that we are done: 
    \begin{itemize}
        \item For each $i\in[ s-1]$, a path from $N_i^3$ to $N_{i+1}^1$; and a path from $N_{s}^3$ to $A_1^1$.
        \item For each $i\in[ s-1]$, a path from $A_i^2$ to $A_{i+1}^1$; and a path from $A_s^2$ to $L_1$
        \item For each $i\in [s-1]$ a path from $R_i$ to $L_{i+1}$; and a path from $R_{s}$ to $L_1$.
    \end{itemize}
Indeed, it is easy to see that the edges $xz_2$ in the nice spiders, and the edge adjacent to $P_1$ and $P_2$ in the cycle extenders will be chords in the created cycle, so we will have $s$ chords. The length of the created cycle is at most $10s\log^{30}n$.

Finally, note that these paths can be found by a greedy procedure. Suppose we want to find the $j$-th path (and note that we only find $3s$ paths). Assuming that each path is of length at most $\log^7 n$, we used at most $\ell=s\cdot \log^{40}n$ vertices, including the gadgets themselves. Since in every step we need to connect sets of size at least $m>\ell \log^{10}n$, we can successfully avoid all previously used vertices with a new path of length at most $\log^7 n$.

The total length of the obtained cycle is at most $10s\log^{30}n$ and we have $s=2^{\log^{1/100}n}$ chords, hence we are done.
\end{proof}
\subsection{Controling high degree vertices}

    \begin{claim}
        Let $R$ be the set of vertices of degree at least $4$ to $L$.
        Then $|R|\leq m^{1/4}$.
    \end{claim}
    \begin{proof}
        
        Suppose $|R|\geq m^{1/4}$; we will show that then there exists a collection of $m^{1/8}/8\geq 2^{\log^{1/100}(n)}$ vertices in $R$ that are roots of vertex disjoint nice spiders, so we would get many gadgets and thus a contradiction by \Cref{lem:connecting gadgets}.

        Let $\mathcal S$ a largest collection of disjoint nice spiders that are $3$-stars with centers in $R$; denote by $R'$ its centers and assume $|R'|\leq m^{1/8}/8$. Now, each vertex $v\in R\setminus R'$ has at most $2 $ neighbours in $L$ outside of $\mathcal S$, as otherwise we get a new nice spider which is in fact a $3$-star rooted at $v$.     
        Thus each $v\in R\setminus R'$ has at least $2$ neighbours in $\mathcal S$. The union of spiders in $\mathcal S$ is of size at most $4m^{1/8}/8$. Since $|R\setminus R'|\geq m^{1/4}/2$, by pigeonhole there is at least one pair of vertices in the union of the nice spiders that is adjacent to the same two vertices in $R\setminus R'$. This gives a $C_4$, a contradiction.
    \end{proof}
\begin{claim}
     $|L|\leq n/m^{1/2}$.
\end{claim}
\begin{proof}
    Otherwise, the number of edges that touch $L$ is at least $nm^{1/2}/2$. On the other hand, since $|R|\leq m^{1/4}$ we have that the number of edges that touch $L$ is at most $|R|n+(n-|R|)4< nm^{1/2}$, a contradiction. 
\end{proof}

\subsection{Maximal collection of gadgets and the structure outside}
    Consider a maximal collection of disjoint gadgets, and recall that the number of them is at most $2^{\log^{1/100}(n)}$ by \Cref{lem:connecting gadgets}. Denote the vertex set of this collection by $W$.
    
    \begin{claim}\label{cl:cleaning}
    Denote $U:=W\cup R\cup L$. There exists a set $B\subset V(H)$ such that graph $G':=H\setminus (U\cup B)$ is an $\frac{1}{4\log^{2}(n)}$-expander. Furthermore, we can chose $B$ of size $|B|\leq 2|U|\log^4 n$ such that $|N_{H\setminus U}(B)|\leq |B|$
    \end{claim}
    \begin{proof}
        Since $|U|\leq 2n/m^{1/2}$, by \Cref{lem:clean for expansion} there is a subset $B$ as required.
    \end{proof}



    

    Consider the $2$-connected components of $G'$. Since $G'$ is an $1/\log^3 n$-expander, it contains a cycle of length at least $|G'|/\log^{13} n$ by \Cref{thm: long cycle}. Let $D$ be the component that contains such a cycle. We can think of the rest of the graph as connected \emph{clusters}, each one attached to one of the vertices of $D$.   
    By expansion, no cluster attached via a vertex to $D$ has size greater than $4\log^2 n$. Indeed if a cluster $D'$ has $4\log^2n\leq |D'|\leq |G'|/2$, then $N_{G'}(D'-v)=\{v\}$ where $v=D\cap D'$; otherwise, if $|D'|>|G'|/2$, then $N(V(G')-D')=\{v\}$ which again is a contradiction as $|D|-1\leq |V(G')-D'|\leq n/2$.\\
    Notice that this bound on the clusters implies that $|D|>|G'|/4\log^2 n>n/8\log^2 n$.
    \begin{claim}
        $D$ is an $\frac{1}{\log^5(n)}$-expander and has at least $3|D|/4$ vertices of degree less than $100.$ 
    \end{claim}
    \begin{proof}
        Consider $X\subset D$ of size at most $|D|/2$. Let $X'$ be the union of all the clusters attached to vertices in $X$. Note 
        that $|V(G')\setminus X'|=|D\setminus X|\geq |D|/2\geq \frac{n}{20\log^2(n)}.$
        
        Now, we have $|N_D(X)|=|N_{G'}(X')|$. If $|X'|\leq |G'|/2$ then $|N_{G'}(X')|\geq |X'|/4\log^{2}n\geq |X|/4\log^2 n$, so we get the required expansion.
        
    Otherwise, if $|X'|\geq |G'|/2$, assume for contradiction that $|N_{G'}(X')|\leq n/\log^5(n)$.
    Consider the set $S:=V(G')\setminus(X'\cup N_{G'}(X'))$. Note that $|S|\geq |D|-|X|-|N_{G'}(X')|\geq |D|/2-n/\log^5(n)\geq n/20\log^2 n$. Furthermore, 
    by definition, all the neighbours of $S$ in $G'$ are in $N_{G'}(X')$, as they cannot be in $X'$, since $S\cap N_{G'}(X')=\emptyset$. Using the expansion in $G'$, we thus get that 
    \[
      n/\log^{5}n<|S|/4\log^{2}n\leq |N_{G'}(S)|\leq |N_{G'}(X')|\leq n/\log^{5}n,
    \]
    a contradiction which completes the proof of the first part of the claim.

        For the second claim, if we assume that at least $|D|/4$ vertices have degree at least $100$, then the average degree is at least $50$, so we get another gadget by applying \Cref{lem:cycle extender} to $D$. Here we note that this is the only and crucial application of this lemma.

    \end{proof}

    \begin{claim}\label{cl:most vertices are good}
        All but at most $n/m^{1/5}<|D|/4$ vertices $v\in G'$ satisfy $d_{G'}(v)>d_H(v)-5$ and have no neighbours in $(R\cup W\cup B)\setminus L$
    \end{claim}
    \begin{proof}
    Since $R\cap V(G')=\emptyset$, each $v\in G'$ has at most $5$ neighbours in $L$. Furthermore, $|N(W\cup R\setminus L)|\leq |R\cup W|m<\sqrt{n}$.
    By \Cref{cl:cleaning}, we have $|N_{G'}(B)|\leq|B|\leq 2n\log^4 n/m^{1/4}$.
    Thus, we are done as every vertex in $G'-(N(W\cup R\setminus L)\cup N_{G'}(B))$ has at least $d_H(v)-5$ neighbours in $G'$.
    \end{proof}

    By  the two claims, we must have at least $|D|/2$ vertices $v\in D$ such that its cluster $D_v$ is non-empty, and such that $D_v$ only contains vertices which satisfy $d_{G'}(v)>d_H(v)-5$ and have no neighbours in $(R\cup W\cup B)\setminus L$.
    Denote by $\mathcal D$ the set of such $v$.
    For each $v\in \mathcal{D}$, chose an arbitrary leaf in the block-cut tree of $G'$ which is contained in the block cut tree of $D_v$ and call $D_v'$ the subgraph of $G'$ to which it corresponds. If $c_v$ is the cut vertex by which $D_v'$ is attached to the rest of the graph, each vertex in $D_v'-c_v$ needs to have $d_G'(v)\geq d_H(v)-5$ neighbours in $D_v'$, so $|D_v'|\geq10^{100}.$   \begin{claim}\label{cl:abc}
        Let $L_1\coloneqq L\setminus W$ (the large degree vertices without the already found gadgets).
        There are no three vertices $a,b,c\in G'$ where $a,b\in D_v'$, for some $v\in D$, so that for each $x\in\{a,b,c\}$ there is a distinct neighbour $y_x\in L_1$.
    \end{claim}
     \begin{proof}
         If such vertices exist, we first find a cycle that contains $a,b$ in $D_v$ because of $2$-connectivity of $D_v'$. Then we find a path from that cycle to $c$ by connectivity of $G'$. Hence there exists a path whose endpoints have neighbours in $L_1$, whose internal point has a neighbour in $L_1$ as well. This clearly creates a new nice spider as the neighbours in $L_1$ are distinct.
     \end{proof}
     
    We will now show that, since we cannot find such three vertices with neighbours in $L_1$, there is a contradiction with the fact that $H$ (our initital graph) is a robust expander.


    

    \subsection{Getting another gadget and completing the proof.}
        By \Cref{cl:abc}, for all sets $D_v'$ there is a set $B_v$ of size at most $3$, such that the whole set $D_v'-B_v$ has at most 2 neighbours in $L\setminus W$. Indeed, take the largest matching from $D_v'-c_v$ to $L\setminus W$. By \Cref{cl:abc} it is of size at most $2$. Denote by $B_v$ the set of matched vertices in $D_v'-c_v$, plus the vertex $c_v$. Clearly $|B_v|\leq 3$ and all vertices in $D_v'-B_v$ have no neighbours in $L\setminus W$ apart from maybe the two matched vertices.     
        Denote $S_v:=D_v'-B_v$.
    
     We distinguish two cases to complete the proof:\\
      \textbf{Case I: If there are at least 
        $2^{\log^{1/7}(n)}$ vertices $v\in \mathcal D$ for which        
        $|S_v|\geq \log^{1/3}(n)$.} \\ 
        Let $S$ be the union of the sets $S_v$ for those vertices, so we have $|S|:=2^{\log^{1/7}(n)}k$ for some $k\geq \log^{1/3}(n)$. 
        Recall that  the neighbourhood in $H$ of each $u\in S_v$ is contained in $V(G')\cup L$. Since $S_v$ only has at most $2$ neighbours in $L\setminus W$, at most $3$ neighbours in $G'$ (those are in $B_v\cup\{c_v\}$), we have $$|N(S)|\leq 5\cdot2^{\log^{1/7}n}+|W|\leq 6\cdot2^{\log^{1/7}n}$$
       On the other hand, by robust expansion,   we have
       \[
       |N_H(S)|\geq |S|\cdot \frac{\varepsilon_1}{\log^2(15|S|)}= \frac{\varepsilon_1 2^{\log^{1/7}(n)}k}{(\log(15k)+\log^{1/7}n)^2 }\geq 10\cdot {2^{\log^{1/7}(n)}}
       \]
       where we used that $k\ge \log^{1/3}n$, thus obtaining a contradiction.
       \\
    \textbf{Case II: At least $|\mathcal D|/2$ vertices in $\mathcal D$ satisfy  $|S_v|\leq \log^{1/3}(n)$.} There are at least $|\mathcal D|/\log^{1/3}n$ of them of the same size $t$, where $d-10\leq t\leq \log^{1/3}n$. Furthermore, among those there are $$\frac{1}{|W|^{5\log^{1/3}n}}\frac{|\mathcal D|}{\log^{1/3}n}\geq \frac{1}{2^{\log^{0.35}n}}\frac{n}{\log^{20}n}\geq t^{10}$$
    which have the exactly the same neighborhood in $W\cap L$ (since the neighbourhood of $S_v$ is at most $5|D_v'|$ in $W\cap L$, because $D_v'\cap R=\emptyset$). Let $I$ be a subset of size $t^{10}$ of such $v$, and let $X=\bigcup_{v\in I}S_v$.
    Recall that each set $S_v$ only has at most $2$ neighbours in $L\setminus W$.

    Thus we have 
   \[
\begin{aligned}
|N_H(X)|
&= 3|I| + \cancelto{0}{|N(X, R \cup W \cup B \setminus L)|} + |N_H(X, L \setminus W)| + |N_H(X, L \cap W)| \\
&\le 3t^{10}+2t^{10} + 5t \le 6t^{10}.
\end{aligned}
\]
On the other hand, by robust expansion we have that $|N_H(X)|\geq t^{11}\frac{\varepsilon_1}{\log^{2}(15t^{11})}> 6t^{10}$, a contradiction, since $t$ is large enough.
    \section{Concluding remarks}
    First we point out that with a bit more effort it is very plausible one could get a smaller constant on the power of $\log(\ell)$ but we opted to not do it to make the paper more readable.  
    
    We believe that it is probably true that a graph with sufficiently high minimum degree has a cycle which spans a linear number of chords. 
    Maybe a first step would be to prove it when the graph is regular. In particular, we conjecture the following.  
    
    \begin{conjecture}
        Let $G$ be a graph with average degree at least $C\log\log(n)$ show that it contains a cycle $C$ on $\ell$ vertices with at least $\ell/2$ chords, for some $\ell\geq 4$.
    \end{conjecture}

     If true this would improve the results of \cite{draganic2024cycles}.
\end{proof}
\bibliographystyle{plain}

\end{document}